\documentclass[12pt]{amsart}

\usepackage{amsfonts}
\usepackage{amssymb}
\usepackage{graphicx}
\usepackage{amsmath,mathtools}
\usepackage{latexsym}
\usepackage{xypic}
\usepackage{mathrsfs}%scrfont}
\usepackage{enumitem} %http://ctan.org/pkg/enumitem
\usepackage{braket}
\usepackage[margin=1.35in]{geometry}
\usepackage{esint}
\usepackage{dsfont}
\usepackage{amsthm}
\usepackage{tikz}

\usepackage{tikz-cd}
\usepackage{hyperref}

\setlist{itemsep=3pt}

\allowdisplaybreaks

\newtheorem{prop}{Proposition}
\newtheorem{theo}[prop]{Theorem}
\newtheorem{lemm}[prop]{Lemma}
\newtheorem{coro}[prop]{Corollary}

\newtheorem*{lemm*}{Lemma}

\theoremstyle{definition}
\newtheorem{defi}[prop]{Definition}

\newtheorem{quest}[prop]{Question}
\newtheorem{rema}[prop]{Remark}

\numberwithin{prop}{section}
 
\theoremstyle{plain}



\newcommand{\RR}{\mathbb{R}}

\newcommand{\TT}{\mathbb{T}}
\newcommand{\ZZ}{\mathbb{Z}}

\newcommand{\cA}{\mathcal A}

\newcommand{\cS}{\mathcal S}

\DeclareMathOperator{\Hom}{Hom}

\DeclareMathOperator{\Ric}{Ric}

\DeclareMathOperator{\imag}{Im}

\DeclareMathOperator{\dist}{dist}

\usepackage{graphicx}

\newcommand{\eps}{\varepsilon}

\title[Covering instability for PSC]{Covering instability for the existence of positive scalar curvature metrics}

\begin{document}

\author{Chao Li}
\address{Courant Institute, New York University, 251 Mercer St, New York, NY 10012, USA}
\email{chaoli@nyu.edu}

\author{Boyu Zhang}
\address{Department of Mathematics, The University of Maryland at College Park, Maryland, 20742, USA}
\email{bzh@umd.edu}

\maketitle

\begin{abstract}
	We show that a closed non-orientable $3$-manifold admits a positive scalar curvature metric if and only if its orientation double cover does; however, for each $4\le n\le 7$, there exist infinitely many smooth non-orientable $n$-manifolds $M$ that are mutually non-homotopy equivalent, such that the orientation double cover of $M$ admits positive scalar curvature metrics, but every closed smooth manifold that is homotopy equivalent to $M$ cannot admit positive scalar curvature metrics. These examples were first introduced by Alpert-Balitskiy-Guth in the study of Urysohn widths.
	
	To prove the nonexistence result, we extend the Schoen-Yau inductive descent approach to non-orientable manifolds. We also discuss band width estimates and the notion of enlargeability for non-orientable PSC manifolds.
\end{abstract}

\section{Introduction}
\label{section.introduction}
%!TEX root = main.tex

The existence of Riemannian metrics with positive scalar curvature (abbreviated as PSC in this article) has been an important topic in geometry and topology. It is often assumed that the manifold under consideration is orientable. Indeed, if a non-orientable manifold has a PSC Riemannian metric, then so does its orientable double cover (by simply taking the covering metric). A natural question is whether the converse statement holds. As we will see from this paper, the answer is positive for $3$-manifolds (see Theorem \ref{theo.3d}), and is negative in higher dimensions. 

\begin{theo}\label{theo.example}
	For each $4\le n\le 7$, there exist infinitely many mutually non-homotopy equivalent smooth closed non-orientable manifolds $M$ with dimension $n$ such that the following hold:
	\begin{enumerate}
		\item The orientation double cover of $M$ admits PSC metrics,
		\item Every closed smooth manifold that is homotopy equivalent to $M$ cannot admit any PSC metric. In particular, $M$ does not admit any PSC metric.
	\end{enumerate}
\end{theo}

Previously, there have been examples where a closed manifold $M$ does not admit any PSC metric while a suitable finite cover of $M$ does. The first such example - an exotic $\RR P^9$ - was discovered by B\'erard Bergery in \cite[Example 9.1]{Bergery1983}. LeBrun\footnote{Manifolds in \cite{LeBrun2003} actually depict a sign change on the Yamabe invariants after taking double cover.} \cite{LeBrun2003} and Hanke--Kotschick--Wehrheim \cite{HankeKotschickWehrheim2003} constructed more examples of this type for orientable $4$--manifolds (see also a subsequent work \cite{AIP2015} and Rosenberg \cite[1.2]{Rosenberg1986conjecture}).  Ruberman--Saveliev \cite[Example 6.4]{ruberman2007dirac} constructed a non-orientable example in dimension 4. We remark here that in all the above examples, the obstructions for PSC metrics depend on the smooth structure of $M$; except for the examples by Bergery which were in dimensions $\ge 9$, all other examples were in dimension 4, and the invariants that obstructed PSC were only defined in dimension 4. It would be interesting to further compare these approaches.

%We emphasize here that orientability is crucial in these examples, in order for index-theoretic invariants (Hitchin's $\alpha$-invariants in \cite{Bergery1983} and Seiberg--Witten in \cite{LeBrun2003,HankeKotschickWehrheim2003}) to be applied to rule out the existence of PSC metrics. In comparison, each example in Theorem \ref{theo.example} has the feature that any manifold homeomorphic to it admits no PSC metric. It would be interesting to further compare these approaches.

\subsection{A brief description of the construction}
We now briefly describe the construction of our examples. We will give full details of the construction and its basic geometric and topological properties in Section \ref{section.construction.X}.

We describe a construction for $n=4$ and $6$. More generally, taking product with a torus $\TT^l$ gives more of such examples, as long as the result is of dimensions $\le 7$. The example described here was first introduced by Alpert--Balitskiy--Guth \cite[Example 11]{ABG2024}, based on a classical example by Gromov \cite[$H_1''$]{Gromov1988width}. 

The original construction in \cite{ABG2024} gives a metric space with very interesting metric-geometric properties (most notably on Urysohn widths). By adding a generic perturbation, this metric space can be modified to a smooth manifold with similar properties. However, in order to understand differential geometric invariants - e.g. scalar curvature - we need to have a more precise control on the manifold. Therefore, instead of directly following the construction in \cite{ABG2024}, we will take a different approach by defining the desired manifold as a piecewise linear (PL) manifold. It is a classical result in geometric topology that every PL manifold with dimension $\le 6$ has a unique smoothing (see \cite[Theorem 2]{milnor2011differential}). Working in the PL category allows us to define the desired manifold explicitly without taking any perturbations. This is crucial for the proof of the existence of PSC metrics on the orientation double cover. 

The manifold is constructed as follows.
Write $n=2k$ with $k=2$ or $3$. Let $Z$ be the subset of $\RR^{2k+1}$ consisting of all points such that at most $k$ coordinates are non-integers. Equivalently, if we endow the unit interval $[0,1]$ with the CW complex structure that has two 0-cells and one 1-cell, endow $[0,1]^{2k+1}$ with the product CW complex structure, write $\RR^{2k+1}$ as the union of $[0,1]^{2k+1} + v$ for all $v\in \ZZ^{2k+1}$ and consider the induced CW complex structure on $\RR^{2k+1}$, then $Z$ equals the $k$--skeleton of $\RR^{2k+1}$. Define $Z' =  Z + (\tfrac 12, \cdots, \tfrac 12)$.

As we will show in Section \ref{section.construction.X}, there exists a triangulation of $\mathbb{R}^{2k+1}$ such that $Z,Z'$ are subcomplexes; moreover, we will show that there exists a PL submanifold $X_0\subset \mathbb{R}^{2k+1}$ with the following properties:

\begin{enumerate}
	\item $\RR^{2k+1}\setminus X_0$ has two connected components $U, U'$,
	\item The closures of $U, U'$ are regular neighborhoods of $Z, Z'$ respectively. 
	\item $X_0$ is invariant under translations by vectors in $\ZZ^{2k+1}$ and $\ZZ^{2k+1}+(\frac12,\dots,\frac12)$. 
\end{enumerate}

Fix an integer $L>0$\footnote{Different choices of $L$ give mutually non-homotopy equivalent examples for Theorem \ref{theo.example} (see Lemma \ref{lem_Euler_X}), and our proof works the same for all such choices.}, consider the vectors 
\[
v_1 = (L,0,\cdots, 0), v_2 = (0,L,0,\cdots,0), \cdots, v_{2k} = (0,\cdots, 0, L, 0),
\]
and
\[
v_{2k+1} = (\tfrac12, \tfrac12, \cdots, \tfrac12, \tfrac12 + L).
\]
Let $G$ be the subgroup of $\RR^{2k+1}$ generated by $v_1,\dots,v_{2k+1}$. Then $G$ is a free abelian group and it acts on $\RR^{2k+1}$ by translations. Define 
\[
X = X_0/G.
\]
We will see in Section \ref{section.construction.X} that $X$ is non-orientable because the translation by $v_{2k+1}$ swaps $Z,Z'$ and hence reverses the orientation of $X_0$. 
When $k=2,3$, define the smooth structure on $X$ to be the unique smoothing of $X$ as a PL manifold. We will show that if $M=X\times \TT ^l$ has dimension $\le 7$, then every closed manifold homotopy equivalent to $M$ admits no PSC metrics, but the orientation double cover of $M$ admits PSC metrics. 

\subsection{Sketch of the proof}
To show that the orientation double cover of $X\times \TT ^l$ admits PSC metrics, let $\hat G$ be the subgroup of  $\RR^{2k+1}$ generated by $v_1,\dots,v_{2k},2v_{2k+1}$, then $\hat G$ is a subgroup of $G$ with index $2$, and $\hat G$ acts on $X_0$ by orientation-preserving PL homeomorphisms. Therefore, the orientation double cover of $X$ is equal to $X_0/\hat G$. Denote the orientation double cover of $X$ by $\hat X$. Then $\hat X$ is the boundary of a regular neighborhood of $Z/\hat G$ in $\RR^{2k+1}/\hat G$. Note that $Z/\hat G$ has codimension $k+1 \ge 3$ in $\RR^{2k+1}/\hat G$. In Lemma \ref{lem_hatX_PSC}, we will use this fact to show that $\hat X$ admits a PSC metric. 

To show that every closed manifold homotopy equivalent to $M=X\times \TT ^l$ does not admit any PSC metric, we observe that the translations of $v_1,\cdots, v_{n-1}$ give rises to integral cohomology classes $\alpha_1,\cdots, \alpha_{n-1}\in H^1(X; \ZZ)$ satisfying
\[\alpha_1\smile \cdots\smile \alpha_{n-1} \ne 0\in H^{n-1}(X;\ZZ).\]
In Section \ref{sec.Schoen.Yau}, we prove in Theorem \ref{theo.cohomology} that this structure of the cohomology ring obstructs the existence of PSC metrics using an extension of the Schoen--Yau inductive descent argument \cite{SY79structure} for non-orientable manifolds. Let us emphasize that it is crucial that we consider integral cohomology instead of cohomology with $\ZZ_2$ coefficients, see Remark \ref{rema.cohomology.Z2}. This result enables us to discuss extensions of Gromov's band width estimates, as well as the notion of enlargeability, to PSC nonorientable manifolds. We carry these out in Section \ref{sec.enlargeability}. We mention here that Gromov--Hanke \cite{GromovHanke} recently considered extensions of the Schoen--Yau inductive descent argument for cohomology groups with other coefficient rings.

\subsection{Connections to macroscopic geometry}
For an integer $d\ge 0$, the Urysohn $d$-width of a metric space $X$, denoted by $UW_d(X)$, measures how well one may approximate $X$ with a $d$-dimensional simplical complex. For any fixed $X$, $UW_d(X)$ is a decreasing function of $d$. In \cite{ABG2024}, it is observed that the manifold $X$ constructed above (with a choice of the parameter $L>0$) has $UW_{n-1}(X) \lesssim L$,  while its orientation double cover satisfies $UW_k(\hat X)\lesssim 1$ (uniformly in $L$). This non-locality for Urysohn widths is our original motivation to study scalar curvature on such manifolds. We refer readers to the recent work of Alpert--Banerjee--Papasoglu \cite{alpert20251urysonwidthcovers} for interesting developments on these problems.

A well-known conjecture by Gromov (see, e.g. \cite{Gromov1988width}) states that there exists a constant $c(n)$ such that $UW_{n-2}(X)\le c(n)$ whenever $(X^n,g)$ has scalar curvature bounded below by $1$. This conjecture is proved when $n=3$ by Liokumovich--Maximo \cite{LiokumovichMaximo} (see also \cite{CLbubbles}) and is open in higher dimensions. Our analysis on the scalar curvature of $X$ and $\hat X$ fits well with Gromov's conjecture: taking the orientation double cover of $X$ simutanously decreases the Urysohn $(n-2)$-width and gives the existence of PSC metrics.

It would be interesting to investigate whether a similar covering instability phenomenon may happen for other curvature conditions - for instance the condition of positive Ricci curvature. Note that the classical metric-geometric properties of $\Ric_g>1$ -  e.g. the Bonnet--Myers' diameter bound (or equivalently the Urysohn $0$-width bound) and the Bishop--Gromov volume comparison - are certainly preserved when passing to a covering space. We remark here that the orientable manifolds (with no PSC metrics) constructed by LeBrun in \cite{LeBrun2003} actually have a double cover admitting a metric with positive Ricci curvature by Sha--Yang \cite{ShaYang}.

\begin{quest}
	Let $X$ be a closed non-orientable manifold and $\hat X$ be the orientation double cover. Suppose $\hat X$ admits a Riemannian metric with positive Ricci curvature. Is it necessarily true that $X$ admits a Riemannian metric with positive Ricci curvature?
\end{quest}

\subsection{Remarks on higher dimensions}

It is an interesting question whether the construction of $X$ in Section \ref{section.construction.X} extends to higher dimensions. Recall that $X$ arises as a PL manifold from a suitably chosen triangulation of $\RR^{2k+1}$ for $n=2k$. When $k \ge 4$, both the nonexistence of PSC metrics on $X$ and the existence of PSC metrics on $\hat X$ are considerably more delicate. We discuss here the difficulties and potential approaches to overcome them.

In our proof of the nonexistence result for $X$, we work in low dimensions to guarantee the regularity of area minimizing hypersurfaces. This proof can be extended to dimensions $\le 11$ ($k\le 5$) with the recent breakthroughs \cite{CMS2023generic,CMS2024improved,CMSW2025generic} on the generic regularity of miniming hypersurfaces. One may also possibly carry out the inductive descent argument with the presence of singularities in light of Schoen--Yau \cite{SY2019singularities}. Note that this proof relies only on the homotopy type of $X$.

The proof that $\hat X$ has a PSC metric is also delicate in higher dimensions. Our definition of the smooth structure on $X$ relies on the fact that every PL manifold with dimension $\le 6$ has a unique smoothing. 
%Since the proof of Theorem \ref{theo.example} only requires the construction of $X$ in dimensions $\le 4$, this is sufficient for the proof of the main result of this paper. Nevertheless, it is a natural question whether a smoothing of $X$ can be defined in higher dimensions. 
In higher dimensions (when $k\ge 4$), the argument of Lemma \ref{lem_hatX_PSC} works verbatim and shows that the orientation double cover $\hat X$ of $X$ admits \emph{at least one} smooth structure which supports a PSC metric. However, it is a priori unclear whether $X$ admits a smoothing whose orientation double cover agrees with the above smooth structure on $\hat X$.

\subsection{Acknowledgment}
The authors are greatly in debt to Larry Guth for stimulating conversations on the subject of this paper, and Hongbin Sun for discussions on $3$-manifolds. We also thank Misha Gromov, Bernhard Hanke, Claude LeBrun and Daniel Ruberman for bringing up related examples. C.L. was partially supported by an NSF grant (DMS-2202343), a Simons Junior Faculty Fellowship and a Sloan Fellowship. B.Z. was partially supported by an NSF grant (DMS-2405271) and a travel grant from the Simons Foundation.

\section{The construction of $X$}
\label{section.construction.X}
%!TEX root = main.tex

We give a careful definition for the manifold $X$ when $n=4$ or $6$ as described in the introduction. 
Since every PL manifold with dimension $\le 6$ has a unique smoothing \cite[Theorem 2]{milnor2011differential}, it suffices to define $X$ as a PL manifold. 

For each positive integer $m$, we introduce a triangulation $\mathcal{T}_m$ of $[0,1]^m$ as follows. Consider the set $\mathcal{S}$ of all sequences  $(p_0,p_1,\dots,p_m)$, such that:
\begin{enumerate}
	\item  $p_i\in \{0,1\}^m$ for all $i$,
	\item  $p_0=(0,\dots,0)$, $p_m = (1,\dots,1)$,
	\item For each $i=0,\dots, m-1$, the point $p_{i+1}$ is obtained from $p_i$ by changing exactly one coordinate from $0$ to $1$.
\end{enumerate}
There is a bijection between $\mathcal{S}_m$ and the set of permutations of $(1,\dots,m)$ described as follows. For $i\in \{1,\dots,m\}$, define  $e_i$ to be the vector in $\RR^{2k+1}$ such that the coordinate indexed by $i$ is $1$ and all the other coordinates are $0$. Suppose $(\sigma_1,\dots,\sigma_m)$ is a permutation of $(1,\dots,m)$, define the corresponding element $(p_0,\dots,p_m)\in \mathcal{S}_m$ by taking $p_i = \sum_{j\le i} e_{{\sigma}_{j}}$. It is straightforward to verify that this yields a bijection between the set of permutations of $(1,\dots,m)$ and $\mathcal{S}_m$. As a consequence, the set $\mathcal{S}_m$ has $m!$ elements. 

Each $(p_0,\dots,p_m)\in \mathcal{S}_m$ satisfies the property that $p_{i+1}-p_i$ are linearly independent for  $i=0,\dots,m-1$, so the points $p_0,\dots,p_m$ span a simplex in $[0,1]^m$. Suppose $(p_0,\dots,p_m)$ corresponds to the permutation $(\sigma_1,\dots,\sigma_m)$ of $(1,\dots,m)$, then the simplex spanned by $(p_0,\dots,p_m)$ consists of all points $(x_1,\dots,x_m)\in [0,1]^m$ such that $x_{\sigma_1}\ge x_{\sigma_{2}}\ge\dots\ge x_{\sigma_{m}}$. As a result, the union of all the $m!$ simplices defined by elements in $\mathcal{S}_m$ cover $[0,1]^m$, and the intersection of each pair of simplices is the simplex spanned by their common vertices. Therefore, the simplices defined by elements in $\mathcal{S}_m$ yields a triangulation of $[0,1]^m$. We denote this triangulation by $\mathcal{T}_m$. By definition, $\mathcal{T}_m$ is a simplicial complex\footnote{It is clear that simplicial complexes can be regarded as CW complexes. Note that in the standard notation, a \emph{cell} of a CW complex is homeomorphic to an open ball, while a \emph{simplex} of a simplicial complex includes its boundary.} that is canonically homeomorphic to $[0,1]^m$. 

We prove several properties of $\mathcal{T}_m$ that will be needed later. In order the state the properties, we need to introduce some additional terminology. Recall that we may view $[0,1]$ as a CW complex with two 0-cells $\{0\}$ and $\{1\}$ and one 1-cell given by $(0,1)$. As a result, the product space $[0,1]^{m}$ admits a product CW complex structure. The product CW complex structure contains exactly one $m$-cell, so it is clearly different from $\mathcal{T}_m$. For each $m'<m$, the number of $m'$ cells of the product CW complex structure is equal to $2^{m-m'}\, {m\choose m'}$. The closure of each $m'$ cell is isometric to $[0,1]^{m'}$, and we call it an \emph{$m'$--face}. An embedding $\iota: [0,1]^{m'}\to [0,1]^m$ is called a \emph{canonical parametrization} of an $m'$--face, if there exist $m'$ coordinates of $[0,1]^m$, such that $\iota$ equals the identity map on these $m'$ coordinates, and $\iota$ equals a constant map to $\{0,1\}^{m-m'}$ on the rest  $m-m'$ coordinates. 

\begin{lemm}
	\label{lem_property_Tm}
	The triangulation $\mathcal{T}_m$ has the following properties.
	\begin{enumerate}
		\item Each face of $[0,1]^m$ is a subcomplex of $\mathcal{T}_m$. 
		\item Suppose $j: [0,1]^{m'}\to [0,1]^m$ is a canonical parametrization of an $m'$--face, then the map $i$ induces an isomorphism of simplicial complexes from $\mathcal{T}_{m'}$ to the restriction of $\mathcal{T}_m$ on $\imag(i)$. 
		\item Suppose $\varphi:[0,1]^m\to [0,1]^m$ is an isometry such that
		 \[
		\varphi(\{0,\dots,0\},\{1,\dots,1\}) = (\{0,\dots,0\},\{1,\dots,1\}),
		\] then $\varphi$ induces an automorphism of $\mathcal{T}_m$. Namely, $\varphi$ maps every simplex of $\mathcal{T}_m$ linearly homeomorphically to a simplex of $\mathcal{T}_m$.
	\end{enumerate}
\end{lemm}

\begin{proof}
	Suppose $\iota: [0,1]^{m'}\to [0,1]^m$ is a canonical parametrization of an $m'$--face. Then for each $(q_0,\dots,q_{m'})\in \mathcal{S}_{m'}$, the sequence $\iota(q_0),\dots, \iota(q_{m'})$ is a sequence of points in $\{0,1\}^m$ such that each point is obtained from the previous one by changing exactly one coordinate from $0$ to $1$. As a consequence, there exists $(p_0,\dots,p_m)\in \mathcal{S}_m$ that contains $(\iota(q_0),\dots, \iota(q_{m'}))$ as a consecutive subsequence.  This shows that for each simplex of $\mathcal{T}_{m'}$, its image under $\iota$ is a simplex of $\mathcal{T}_m$. This proves Properties (1), (2). 
	
	To prove (3), note that the top-dimensional simplices of $\mathcal{T}_m$ are given by 
	\[
	\{(x_1,\dots,x_m)\in [0,1]^m\mid x_{\sigma_1}\ge x_{\sigma_2}\ge \dots \ge x_{\sigma_m}\}
	\]
	for permutations $(\sigma_1,\dots,\sigma_m)$ of $(1,\dots,m)$. Since 
	\[
	\varphi(\{0,\dots,0\},\{1,\dots,1\}) = (\{0,\dots,0\},\{1,\dots,1\}),
	\]
	the map $\varphi$ is either a permutation on the coordinates, or a permutation of coordinates composed with the map that takes $(x_1,\dots,x_m)$ to $(1-x_1,\dots,1-x_m)$. In both cases, the map $\varphi$ takes every $m$--dimensional simplex of $\mathcal{T}_m$ linearly homeomorphically to an $m$--dimensional simplex of $\mathcal{T}_m$. Since every lower dimensional simplex of $\mathcal{T}_m$ is a face of an $m$--dimensional simplex, this shows that $\varphi$ is an automorphism of  $\mathcal{T}_m$.
\end{proof}

Now suppose $C$ is an $m$--dimensional cube in a Euclidean space, and suppose $v,v'$ is a pair of opposite vertices of $C$. Let $\mathcal{T}_{v,v'}$ be the triangulation of $C$ defined as follows: Fix a linear homeomorphism $\varphi_{v,v'}$ from $C$ to $[0,1]^m$ such that $v,v'$ are mapped to $(0,\dots,0)$ and $(1,\dots,1)$. The triangulation $\mathcal{T}_{v,v'}$ is defined to be the unique triangulation of $C$ such that $\varphi_C$ induces an isomorphism from $\mathcal{T}_{v,v'}$ to $\mathcal{T}_m$. By Lemma \ref{lem_property_Tm} (3), we know that $\mathcal{T}_{v,v'}$ does not depend on the ordering of $v,v'$ or the choice of $\varphi_{v,v'}$. 
We also record the following property for later reference.
\begin{lemm}
	\label{lem_Tvv'_face}
	Suppose $C$ is a cube in a Euclidean space, $v,v'$ is a pair of its opposite vertices. Suppose $D$ is a face of $C$, and $w,w'$ are the images of the orthogonal projections of $v,v'$ to $D$. Let  $\mathcal{T}_{v,v'}, \mathcal{T}_{w,w'}$ be the triangulations of $C, D$ defined as above. Then the restriction of $\mathcal{T}_{v,v'}$ to $D$ is equal to $\mathcal{T}_{w,w'}$.
\end{lemm}
\begin{proof}
	This is an immediate consequence of Lemma \ref{lem_property_Tm} (2).
\end{proof}

We define a triangulation on $\mathbb{R}^{2k+1}$ as follows. First, write $\mathbb{R}^{2k+1}$ as the union of cubes given by translations of $[0,1/2]^n$ by vectors in $(\frac12\mathbb{Z})^n$. For each such cube, there exist a unique vertex $v\in \mathbb{Z}^n$ and a unique vertex $v'\in \mathbb{Z}^n + (1/2,\dots,1/2)$. Moreover, $v,v'$ are opposite vertices of the cube. Triangulate the cube by $\mathcal{T}_{v,v'}$. By Lemma \ref{lem_Tvv'_face}, when two such cubes intersect at a face, the restrictions on the intersection are the same. Therefore, the triangulations of cubes glue together to define a triangulation of $\mathbb{R}^{2k+1}$. We denote this triangulation of $\RR^{2k+1}$ by $\mathcal{T}$. 

To introduce the definition of $X_0$, we recall the following terminologies from PL topology.

\begin{defi}
	Suppose $\mathcal{C}$ is a simplicial complex, and $\mathcal{C}_0$ is a subcomplex. 
	\begin{enumerate}
		\item The \emph{simplicial neighborhood} of $\mathcal{C}_0$ in $\mathcal{C}$ is the union of all simplices of $\mathcal{C}$ that have non-trivial intersections with $\mathcal{C}_0$. 
		\item We say that $\mathcal{C}_0$ is \emph{full} in $\mathcal{C}$, if for every simplex $s$ such that the vertices of $s$ are all contained in $\mathcal{C}_0$, we have that $s$ is contained in $\mathcal{C}_0$. 
	\end{enumerate}
\end{defi}

By definition, simplicial neighborhoods are subcomplexes. Note that the concepts of \emph{simplicial neighborhoods} and \emph{full subcomplexes} depend not only on the underlying topological spaces but also on the triangulations.

The following result follows from the proof of the existence of regular neighborhoods in PL topology. See, for example, \cite[Theorem 2.11(1)]{hudson1969piecewise}.
\begin{prop}
	\label{prop_regular_nbhd}
	Suppose $\mathcal{C}$ is a PL manifold and $\mathcal{C}_0$ is a full subcomplex. Let $\mathcal{C}', \mathcal{C}_0'$ be the barycentric subdivisions of $\mathcal{C}, \mathcal{C}_0$ respectively. Then the simplicial neighborhood of $\mathcal{C}_0'$ in $\mathcal{C}'$ is a regular neighborhood of $\mathcal{C}_0'$.
\end{prop}

By definition, if $P$ is a polygon in a PL manifold $M$, then $N\subset M$ is called a \emph{regular neighborhood} of $P$ in $M$, if $N$ is a closed neighborhood of $P$ that is a codimension-0 PL submanifold of $M$ and collapses to $P$ (for the definition of \emph{collapsing}, see \cite[Section II.1]{hudson1969piecewise}). Every pair of regular neighborhoods of $P$ are related to each other by a PL isotopy relative to $P\cup \partial M$ (see \cite[Theorem 1.6.4(3)]{rushing1973topological}).

In order to apply Proposition \ref{prop_regular_nbhd} in our construction, we prove the following lemma about $Z$ and $Z'$.
\begin{lemm}
	\label{lem_Z_full}
	Let $Z,Z'\subset \RR^{2k+1}$ be as above. Then $Z, Z'$ are both full with respect to the triangulation $\mathcal{T}$. 
\end{lemm}
\begin{proof}
	Define $Z_0\subset[0,1]^{2k+1}$ to be the union of all faces of $[0,1]^{2k+1}$ with dimensions $\le k$ and containing $(0,\dots,0)$. We only need to show that $Z_0$ is full in $[0,1]^{2k+1}$ with respect to $\mathcal{T}_{2k+1}$. 
	
	The set $Z_0$ consists of all points in $[0,1]^{2k+1}$ such that at least $k+1$ coordinates are equal to $0$. 
	Suppose $p_0,\dots,p_s$ are the vertices of a simplex of $\mathcal{T}_{2k+1}$. By the definition of $\mathcal{T}_{2k+1}$, we have $p_i \in \{0,1\}^{2k+1}$ for all $i$, and we can order the vertices so that each coordinate of $p_{i+1}$ is no less than the corresponding coordinate of $p_{i}$.  If $p_0,\dots,p_s$ are all contained in $Z_0$, then $p_s$ has at least $k+1$ coordinates that are equal to $0$, so $p_0,\dots,p_s$ all have zero values on these coordinates. Therefore, the simplex spanned by $p_0,\dots,p_s$ is contained in $Z_0$. 
\end{proof}

\begin{defi}
	Let $\mathcal{T}'$ be the barycentric subdivision of $\mathcal{T}$. 
	Let $N(Z)\subset \RR^{2k+1}$ be the simplicial neighborhood of $Z$ with respect to $\mathcal{T}'$. Let $N(Z')\subset \RR^{2k+1}$ be the simplicial neighborhood of $Z'$ with respect to $\mathcal{T}'$.
\end{defi}

By Lemma \ref{lem_Z_full} and Proposition \ref{prop_regular_nbhd}, we know that $N(Z), N(Z')$ are regular neighborhoods of $Z, Z'$. In particular, the boundaries of $N(Z)$ and $N(Z')$ are PL submanifolds of $\RR^{2k+1}$. We now show that $N(Z)$ and $N(Z')$ have the same boundary. 

\begin{lemm}
	\label{lem_ZZ'_dual_in_T}
	Suppose $p_0,\dots,p_{2k+1}$ are the vertices of a $(2k+1)$--simplex of $\mathcal{T}$. Then there exists an ordering of the vertices such that $p_0,\dots,p_{k}$ are contained in $Z$, and $p_{k+1},\dots,p_{2k+1}$ are contained in $Z'$.
\end{lemm}
\begin{proof}
	Define $Z_0\subset[0,1]^{2k+1}$ to be the union of all faces of $[0,1]^{2k+1}$ with dimensions $\le k$ and containing $(0,\dots,0)$. 
	Define $Z_0'\subset [0,1]^{2k+1}$ to be the union of all faces with dimensions $\le k$ and containing $(1,\dots,1)$. 
	We only need to show that if $p_0,\dots,p_{2k+1}$ are the vertices of a $(2k+1)$--simplex of $\mathcal{T}_{2k+1}$, then there exists an ordering such that $p_0,\dots,p_{k}$ are contained in $Z_0$, and $p_{k+1},\dots,p_{2k+1}$ are contained in $Z_0'$.
	
	By the definition of $\mathcal{T}_{2k+1}$, one may order $p_0,\dots,p_{k+1}$ so that $p_{i+1}$ is obtained from $p_i$ by changing exactly one coordinate from $0$ to $1$. Note that a point $p\in [0,1]^{2k+1}$ is contained in $Z_0$ if and only if at least $k+1$ of the coordinates of $p$ are equal to $0$, and that $p$ is contained in $Z_0'$ if and only if at least $k+1$ of the coordinates of $p$ are equal to $1$. As a result, it is clear that $p_0,\dots,p_k$ are contained in $Z_0$, and $p_{k+1},\dots,p_{2k+1}$ are contained in $Z_0'$.
\end{proof}

\begin{lemm}
	$\partial N(Z) = \partial N(Z')$.
\end{lemm}

\begin{proof}
	Let $s$ be a $(2k+1)$--simplex of $\mathcal{T}$. We use Lemma \ref{lem_ZZ'_dual_in_T} to show that $\partial N(Z)\cap s=\partial N(Z')\cap s$. Since $\RR^{2k+1}$ equals the union of all $(2k+1)$--simplices of $\mathcal{T}$, the desired result follows. 
	
	The vertices of the barycentric subdivision of $s$ are given by the barycenters of the faces of $s$. Suppose $t_1, t_2$ are the barycenters of $s_1,s_2$, where $s_1,s_2$ are faces of $s$, we write ``$t_1<t_2$'' if $s_1$ is a proper subset of $s_2$; we write ``$t_1\le t_2$'' if $t_1<t_2$ or $t_1=t_2$. Then each sequence of barycenters of the form  $t_1<t_2<\dots<t_u$ corresponds to a simplex of the barycentric subdivision of $s$. A simplex given by $t_1<\dots< t_u$ is contained in $Z$ if and only if $t_u\in Z$; it has a non-trivial intersection with $Z$ if and only if $t_1\in Z$; it is contained in $N(Z)$ if and only if there exists $t_0\le t_1$ such that $t_0\in Z$. 
	
	By Lemma \ref{lem_ZZ'_dual_in_T}, suppose $t_1$ is the barycenter of a face of $s$, there exists $t_0\le t_1$ such that $t_0\in Z$ if and only if $t_1\notin Z'$. 
	As a result, the simplex spanned by $t_1<\dots< t_u$ is contained in $\partial N(Z)\cap s$ if and only if $t_1\notin Z$ and $t_1\notin Z'$. Similarly, the same characterization holds for $\partial N(Z')\cap s$. Hence $N(Z)\cap s = N(Z')\cap s$. 
\end{proof}

We define $X_0 = \partial N(Z) = \partial N(Z')$. By the definition of regular neighborhoods, we know that $X_0$ is a PL submanifold of $\RR^{2k+1}$. It is also clear from the definition that $X_0$ is invariant under translations by vectors in $\ZZ^{2k+1}$ and $\ZZ^{2k+1}+(\frac12,\dots,\frac12)$.

Fix an integer $L>0$, let 
\[
v_1 = (L,0,\cdots, 0), v_2 = (0,L,0,\cdots,0), \cdots, v_{2k} = (0,\cdots, 0, L, 0),
\]
and
\[
v_{2k+1} = (\tfrac12, \tfrac12, \cdots, \tfrac12, \tfrac12 + L).
\]
Let $G$ be the (free) subgroup of $\RR^{2k+1}$ generated by $v_1,\dots,v_{2k+1}$. Then $G$ acts on $\RR^{2k+1}$ by translations and $X_0$ is an invariant subset of the $G$--action. 
Define 
\[
X = X_0/G.
\]

Note that the translations by $v_1,\dots,v_{2k}$ preserve $Z,Z'$ while the translation by $v_{2k+1}$ swaps $Z$ and $Z'$. So the translations by $v_1,\dots,v_{2k}$ preserve the orientation of $X_0$ and the translation by $v_{2k+1}$ reverses the orientation of $X_0$.
Let $\hat G$ be the (free) subgroup of $\RR^{2k+1}$ generated by $v_1,\dots,v_{2k}, 2v_{2k+1}$. Then the orientation double cover of $X$ is $X/\hat G$. Let $\hat X$ denote the orientation double cover of $X$.  

We record the following properties of $X$ and $\hat X$ for later reference. 
\begin{lemm}
	\label{lem_hatX_PSC}
	Suppose $k=2$ or $3$ so that both $X$ and $\hat X$ have unique smoothings. Then the smooth manifold $\hat X$ admits a PSC metric.
\end{lemm}

\begin{proof}
	Note that $\hat X = X/\hat G$ is the boundary of a regular neighborhood $N(Z)/\hat G$ of $Z/\hat G$ in $\RR^{2k+1}/\hat G \cong \TT^{2k+1}$. Since $Z$ is a $k$--complex, its regular neighborhood has a PL handle decomposition by handles of indices $\le k$. As a result, its boundary $\hat X$ can be obtained from a disjoint union of $S^{2k}$'s by a sequence of PL surgeries along $\le k-1$ dimensional spheres. 
	Note that by transversality, every PL embedding of a sphere with dimension $\le k-1$ in a $2k$ dimensional smooth manifold can be PL isotopic to a smooth embedding. Therefore, $\hat X$ has a smoothing that is diffeomorphic to a manifold obtained from a disjoint union of $S^{2k}$'s by a sequence of smooth surgeries along $\le k-1$ dimensional spheres. Since $k\ge 2$, all surgeries have codimension $\ge 3$, so this smoothing admits a metric with PSC. Since the smoothing of $\hat X$ is unique, the result is proved.
\end{proof}

\begin{lemm}
	\label{lem_X_torus_deg}
	There exists a continuous map from $X$ to $\TT^{2k}$ whose $\ZZ_2$ degree is nonzero.  
\end{lemm} 

\begin{proof}
	Note that $X\subset \RR^{2k+1}/G \cong  \TT^{2k+1}$. Let $c$ be the image in $\TT^{2k+1}$ of the straight line segment from $(0,\dots,0)$ to $v_{2k+1}$ in $\RR^{2k+1}$. Then $c$ is a circle. 
	
	The mod 2 intersection number of $X$ and $c$ is odd, because the fundamental domain of the preimage of $c$ in $\RR^{2k+1}$ goes from one side of $X_0$ to the other side of $X_0$.
	As a result, the composition of the inclusion $X\hookrightarrow \RR^{2k+1}/G$ with the quotient map  
	\[
	\RR^{2k+1}/G \to \RR^{2k+1}/\langle v_1,\dots, v_{2k}, \RR v_{2k+1}\rangle  \cong \TT^{2k}
	\]
	has a nonzero $\ZZ_2$ degree. Here, $\langle v_1,\dots, v_{2k}, \RR v_{2k+1}\rangle$ denotes the subgroup of $\RR^{2k+1}$ generated by $v_1,\dots, v_{2k}, \RR v_{2k+1}$.
\end{proof}

\begin{lemm}
	\label{lem_Euler_X}
The Euler number of $X$ is equal to $L^{2k}(L+1)\big(\sum_{i=0}^k (-1)^i {2k+1\choose i}\big)$.
\end{lemm}
\begin{proof}
	Recall that $\RR^{2k+1}$ has a CW complex structure defined by decomposing it as the union of $[0,1]^{2k+1}+v$ for all $v\in \ZZ^{2k+1}$. 
	This induces a CW complex structure on $Z/\hat G$. It is straightforward to verify that the CW complex structure on $Z/\hat G$ has dimension $k$, and the number of $i$--dimensional cells ($i\le k$) is equal to $L^{2k}(L+1){2k+1\choose i}$. So 
	\[\chi(Z/\hat G) = L^{2k}(L+1)\Big(\sum_{i=0}^k (-1)^i {2k+1\choose i}\Big).\]
	
	The manifold $\hat X$ bounds a regular neighborhood $N(Z/\hat G)$ of $Z/\hat G$. The double of $N(Z/\hat G)$ is a closed manifold with dimension $2k+1$, so its Euler number is zero. 
	Hence we have
	\[
	2 \chi (N(Z/\hat G)) - \chi(\hat X) = 0.
	\]
	Therefore,
	\[
	\chi(X) = \tfrac12 \chi(\hat X) = \chi (N(Z/\hat G)) = \chi(Z/\hat G),
	\]
	and the lemma is proved. 
\end{proof}

\section{Schoen--Yau descent for non-orientable manifolds}
\label{sec.Schoen.Yau}
%!TEX root = main.tex

In this section, we prove that the manifold $X$ constructed in Section \ref{section.construction.X} does not admit any PSC metric. In fact, we will prove the following more general result, which is a generalization of the Schoen--Yau inductive descent argument to non-orientable manifolds. 

\begin{theo}\label{theo.cohomology}
	Suppose $2\le n\le 7$, and $M^n$ is a closed, possibly non-orientable manifold with dimension $n$. Suppose there exist $\alpha_1,\cdots, \alpha_{n-1}\in H^1(M^n;\ZZ)$ such that
	\begin{equation}\label{eq.assumption.cohomology}
		\alpha_1 \smile \cdots\smile \alpha_{n-1} \ne 0\in H^{n-1}(M^n;\ZZ).
	\end{equation}
	Then $M^n$ does not admit any Riemannian metric with positive scalar curvature.
\end{theo}

\begin{rema}
	If $M^n$ is orientable, then the theorem follows from the classical Schoen--Yau descent \cite{SY79structure}. Our result extends Schoen--Yau's result to the case of non-orientable manifolds. 
\end{rema}

\begin{rema}\label{rema.cohomology.Z2}
	It is important to have the assumption \eqref{eq.assumption.cohomology} on integral (instead of $\ZZ_2$) cohomology in Theorem \ref{theo.cohomology}. Indeed, $\RR P^{n}$ admits $\alpha\in H^1(\RR P^n; \ZZ_2)$ such that $\alpha^n\ne 0\in H^n(\RR P^n; \ZZ_2)$ and it obviously carries a PSC metric. This contrasts sharply with \cite{Guth2010minimal}, where Guth established the macroscopic analogue of the Schoen--Yau inductive descent for possibly nonorientable manifolds satisfying the analogous cohomology assumption \eqref{eq.assumption.cohomology} with $\ZZ_2$ coefficients. At a technical level, assumption \eqref{eq.assumption.cohomology} with $\ZZ$ coefficients guarantees the two-sidedness of an area-minimizing hypersurface.
\end{rema}

As an immediate corollary of Theorem \ref{theo.cohomology}, we have the following:
\begin{coro}
	\label{cor_odd_deg_torus}
	Suppose $2\le n \le 7$ and $M^n$ is a closed, possibly non-orientable manifold with dimension $n$. Also assume that there exists a continuous map $f:M\to \TT^n$ with a nonzero $\ZZ_2$ degree. Then $M^n$ does not admit any Riemannian metric with positive scalar curvature.
\end{coro}
\begin{proof}
	Let $\beta_1,\dots,\beta_n\in H^1(\TT^n;\ZZ)$ such that $\beta_1\smile\dots\smile\beta_n$ equals a generator of $H^n(\TT^n;\ZZ)$. By the assumptions, $f^*: H^{n}(\TT^{n};\ZZ_2)\to H^{n}(M;\ZZ_2)$ is injective. Consider the following commutative diagram 
	\[\begin{tikzcd}
		H^{n}(\TT^{n};\ZZ)\arrow{r} \arrow{d}{f^*}& H^{n}(\TT^{n};\ZZ_2) \arrow{d}{f^*}\\
		H^{n}(M^n;\ZZ) \arrow{r} & H^{n}(M^n;\ZZ_2),
	\end{tikzcd}
	\]
	where the horizontal maps are given by the surjective homomorphism $\ZZ\to \ZZ_2$ on coefficients. Then the image of $\beta_1\smile \dots\smile\beta_n$  in $H^{n}(M^n;\ZZ_2)$ is non-zero, so 
	\[
	f^*\beta_1\smile\dots\smile f^*\beta_n \neq 0 \in H^{n}(M^n;\ZZ).
	\]
	Therefore, $M^n$ satisfies the assumption of Theorem \ref{theo.cohomology}. Hence $M^n$ does not admit any PSC metric. 
\end{proof}

Combining Corollary \ref{cor_odd_deg_torus} and the earlier results, we can now finish the proof of Theorem \ref{theo.example}. 

\begin{proof}[Proof of Theorem \ref{theo.example} assuming Theorem \ref{theo.cohomology}]
	Let $X$ be the manifold constructed in Section \ref{section.construction.X} for $k=2$ or $3$. By comparing the top-dimensional homology groups with $\ZZ_2$ coefficients, every closed manifold $X'$ homotopy equivalent to $X$ is of dimension $n=2k$. Then by Lemma \ref{lem_X_torus_deg} and Corollary \ref{cor_odd_deg_torus}, the manifold $X'$ does not admit any metric with PSC. On the other hand, by Lemma \ref{lem_hatX_PSC}, the orientation double cover of $X$ admits PSC metrics. This gives the examples for $n=4$ and $6$. 
	
	To construct the examples for $n=5$ and $7$, note that if $M^n$ admit a map to $\TT^n$ with a nonzero $\ZZ_2$ degree, then $M^n\times S^1$ admit a map to $\TT^{n+1}$ with a nonzero $\ZZ_2$ degree. Therefore, the examples for $n=5$ and $7$ can be given by taking the product of $S^1$ with the example in dimension $n-1$. Alternatively, let $X^4$ be the example in dimension $4$, then $X^4\times (S^1)^{n-4}$ are desired examples in dimensions $5\le n\le 7$. 
	
	Recall that the construction of $X$ depends on the choices of a positive integer $L$. 
	By Lemma \ref{lem_Euler_X}, the total rank of the homology groups of $X^4\times (S^1)^{n-4}$ diverges to infinity as $L\to \infty$, so we obtained infinitely many examples that are mutually non-homotopy equivalent. 
\end{proof}

The rest of this section proves Theorem \ref{theo.cohomology}. We apply induction on $n$. When $n=2$, the only closed non-orientable manifold with positive curvature is $\RR P^2$. But $H^1(\RR P^2;\ZZ) = 0$, so the conclusion holds. Assume now that the conclusion holds for $n-1$.

If $M^n$ is orientable, then the result follows from the classical theorem of Schoen--Yau.  We assume without loss of generality that $M^n$ is non-orientable. For the sake of contradiction, assume $M^n$ has a PSC metric $g$. Let $\hat M^n$ be the orientation double cover of $M^n$, let $\pi:\hat M^n\to M^n$ be the covering map, and let $\hat g$ be the pull-back of $g$ to $\hat M^n$. 

\begin{lemm}
	\label{lem_pi_inj_H1}
	The pull back map $\pi^*:H^1(M^n;\ZZ)\to H^1(\hat M^n;\ZZ)$ is injective.
\end{lemm}

\begin{proof}
	By the universal coefficient theorem, $H^1(M^n; \ZZ) \cong \Hom(H_1(M^n;\ZZ),\ZZ)$. Since $H_1$ is the abelianization of $\pi_1$, and $\ZZ$ is abelian, we have a natural isomorphism $\Hom(H_1(M^n;\ZZ),\ZZ ) \cong \Hom(\pi_1(M^n),\ZZ)$. This yields a natural isomorphism
	\[H^1(M^n,\ZZ) \cong \Hom(\pi_1(M^n),\ZZ).\]
	
	Consider the commutative diagram
	\[\begin{tikzcd}
		H^1(M^n,\ZZ) \arrow{r}{\pi^*} \arrow{d}{\cong}& H^1(\hat M^n, \ZZ) \arrow{d}{\cong}\\
		\Hom(\pi_1(M^n),\ZZ) \arrow{r}{\pi^*} & \Hom(\pi_1(\hat M^n),\ZZ),
	\end{tikzcd}
	\]
	where the vertical maps are the natural isomorphisms defined above.
	The image of $\pi_1(\hat M^n)$ in $\pi_1(M^n)$ is a subgroup of $\pi_1(M^n)$ with index $2$. So, a homomorphism $\rho:\pi_1(M^n)\to \ZZ$ is trivial if and only if its restriction to the image of $\pi_1(\hat M^n)$ in trivial. This implies that the map $\pi^*: \Hom(\pi_1(M^n),\ZZ) \to \Hom(\pi_1(\hat M^n),\ZZ)$ is injective, so $\pi^*:H^1(M^n;\ZZ)\to H^1(\hat M^n;\ZZ)$ is injective.
\end{proof}

Let $\alpha_1,\dots,\alpha_{n-1}$ be given by the assumptions of Theorem \ref{theo.cohomology}. By Lemma \ref{lem_pi_inj_H1}, we know that $\pi^*\alpha_{n-1}\ne 0\in H^1(\hat M^n,\ZZ)$. 

Let $\gamma$ be the non-trivial deck transformation on $\hat M^n$, then $\gamma$ is an orientation-reversion isometry of $(\hat M^n, \hat g)$. 

Fix an orientation of $\hat{M}^n$.
Let $PD(\pi^*\alpha_{n-1})\in H_{n-1}(\hat{M}^n;\ZZ)$ denote the Poincar\'e dual of $\pi^*\alpha_{n-1}$. Since $\gamma$ is an orientation-reversing diffeomorphism of $\hat M^n$ and $\gamma^*(\pi^*\alpha_{n-1}) = \pi^*\alpha_{n-1}$, we have
\[
\gamma_* (PD(\pi^*\alpha_{n-1})) = - PD(\pi^*\alpha_{n-1}). 
\]

Consider the set
\[\cS = \{\hat \Sigma\in PD(\pi^*\alpha_{n-1}) \text{ is an integral current}\mid \gamma(\hat \Sigma) = -\hat \Sigma \}.\]
Note that $\cS$ is closed under current convergence: indeed, if $\hat \Sigma_j \in \cS$ converges to $\hat \Sigma$, then $\hat \Sigma$ is integral; moreover, we have that for any continuous $(n-1)$-form $\omega$, $\int_{\hat \Sigma_j} (\gamma^*\omega+\omega)=0$, so $\int_{\hat \Sigma} (\gamma^*\omega + \omega) = 0$. Thus, $\hat \Sigma\in \cS$. The set $\mathcal{S}$ is also non-empty. To find an element in $\mathcal{S}$, let $f:M\to S^1$ be a smooth map such that the pull-back of a generator of $H^1(S^1;\ZZ)$ to $M$ equals $\alpha_{n-1}$, let $p\in S^1$ be a regular value of $f$, then $(f\circ\pi)^{-1}(p)$ can be endowed with an orientation so that it defines an integral current in $\mathcal{S}$. 

Therefore, by minimizing the $(n-1)$-dimensional volume among elements in $\cS$, we obtain a smooth oriented area-minimizing hypersurface $\hat \Sigma^{n-1}$ satisfying $\gamma(\hat \Sigma^{n-1}) = -\hat\Sigma^{n-1}$, where $-\hat\Sigma^{n-1}$ denotes $\hat \Sigma^{n-1}$ with the reversed orientation. Let $\Sigma^{n-1}\subset M^n$ be the image of  $\hat \Sigma^{n-1}$ under $\pi$. 

Note that $\Sigma^{n-1}$ is two-sided. In fact, since $\gamma(\hat \Sigma^{n-1}) = - \hat \Sigma^{n-1}$ and $\gamma$ reverses the orientation of $\hat M^n$, the map $\gamma$ preserves the orientation of the normal bundle of $\hat \Sigma^{n-1}$. Therefore, the normal bundle of $\Sigma^{n-1}$ is oriented. 

By construction, $\hat \Sigma^{n-1}$ is stable among all deformations inside the set $\cS$, so $\Sigma^{n-1}$ is stable among all deformations.  
The stability inequality together with the Schoen-Yau rearrangement \cite{SY79} on the two-sided hypersurface $\Sigma^{n-1}$ give that
\[\int_{\Sigma^{n-1}} |\nabla f|^2 - \frac12 (R_{M^n} - R_{\Sigma^{n-1}} + |A|^2)f^2 \ge 0,\]
for all $f\in C^1(\Sigma^{n-1})$. If $n=3$, this together with Gauss-Bonnet imply that every connected component of $\Sigma^{2}$ is diffeomorphic to $\RR P^2$ or $S^2$. If $n>3$, this implies that $\lambda_1(-\Delta + \tfrac{n-3}{4(n-2)}R_g)>0$ on $\Sigma^{n-1}$, so $g|_{\Sigma^{n-1}}$ is conformal to a positive scalar curvature metric.

In the following, we show that $\Sigma^{n-1}$ satisfies the assumptions of Theorem \ref{theo.cohomology}, so we may apply the induction hypothesis on $\Sigma^{n-1}$.

Let $\nu\Sigma^{n-1}$ be the normal bundle of $\Sigma^{n-1}$ in $M^n$. Since  $\nu\Sigma^{n-1}$ is an oriented line bundle, it has a Thom class
$e \in H^1(\nu\Sigma^{n-1}, \nu\Sigma^{n-1}\setminus \Sigma^{n-1};\ZZ)$. Let $e|_{M^n}$ denote the image of $e$ in $H^1(M^n;\ZZ)$
under the composition map 
\[
H^1(\nu\Sigma^{n-1}, \nu\Sigma^{n-1}\setminus \Sigma^{n-1};\ZZ)  \cong H^1(M^n,M^n\setminus \Sigma^{n-1}; \ZZ) \to H^1(M^n;\ZZ),
\]
where the first isomorphism is given by excision, and the second map is the pull back map of the inclusion $(M^n,\emptyset)\subset (M^n,M^n\setminus \Sigma^{n-1})$.
\begin{lemm}
	We have $e|_{M^n} = \alpha_{n-1}$ in $H^1(M^n;\ZZ)$. 
\end{lemm}

\begin{proof}
	Let $\nu\hat{\Sigma}^{n-1}$ be the oriented normal bundle of $\hat{\Sigma}^{n-1}$.
	Let $\hat{e} \in H^1(\nu\hat{\Sigma}^{n-1},\nu\hat{\Sigma}^{n-1}\setminus \hat{\Sigma}^{n-1};\ZZ)$ be the Thom class of $\nu\hat{\Sigma}^{n-1}$, and let $\hat{e}|_{\hat{M}^{n}}$ denote the image of $\hat{e}$
	under the composition map 
	\[
	H^1(\nu\hat{\Sigma}^{n-1}, \nu\hat{\Sigma}^{n-1}\setminus \hat{\Sigma}^{n-1};\ZZ)  \cong H^1(\hat{M}^n,\hat{M}^n\setminus \hat{\Sigma}^{n-1}; \ZZ) \to H^1(\hat{M}^n;\ZZ).
	\]
	
	By the naturality of Thom class, the pull-back of $e$ equals $\hat {e}$, so $\pi^*(e|_{M}) = \hat e|_{\hat{M}}$. By the relations of Poincar\'e duality with Thom class, we know that $\hat e|_{\hat{M}} = PD([\hat{\Sigma}])$, so $ \hat e|_{\hat{M}} = \pi^*(\alpha_{n-1})$. As a consequence, $\pi^*(e|_{M})= \pi^*(\alpha_{n-1})$. By Lemma \ref{lem_pi_inj_H1}, $\pi^*$ is injective on $H^1(-;\ZZ)$, so  $e|_{M^n} = \alpha_{n-1}$.
\end{proof}

\begin{lemm}
	\label{lem_Sigma_coh_cup}
	$(\alpha_1|_{\Sigma^{n-1}})\smile \cdots \smile (\alpha_{n-2}|_{\Sigma^{n-1}}) \ne 0\in H^{n-2}(\Sigma^{n-1};\ZZ).$
\end{lemm}

\begin{proof}
	Let $p: \nu\Sigma^{n-1} \to \Sigma^{n-1}$ be the projection of the normal bundle of  $\Sigma^{n-1}$ to the base. 
	To simplify notation, we also use $e$ to denote the image of $e\in H^1(\nu\Sigma^{n-1}, \nu\Sigma^{n-1}\setminus \Sigma^{n-1};\ZZ)$ under the excision isomorphism 
	\[
	H^1(\nu\Sigma^{n-1}, \nu\Sigma^{n-1}\setminus \Sigma^{n-1};\ZZ)\cong H^1(M^n,M^n\setminus \Sigma^{n-1};\ZZ).
	\]
	Consider the commutative diagram
	\[ \begin{tikzcd}
		H^{n-2}(M^n;\ZZ) \arrow{r}{\smile e} \arrow[swap]{d}{} & H^{n-1}(M^n,M^n\setminus \Sigma^{n-1};\ZZ) \arrow{d}{\cong} \arrow{r} & H^{n-1}(M^n;\ZZ) \\%
		H^{n-2}(\nu \Sigma^{n-1};\ZZ) \arrow{r}{\smile e}& H^{n-1}(\nu\Sigma^{n-1}, \nu\Sigma^{n-1}\setminus \Sigma^{n-1};\ZZ)
	\end{tikzcd}
	\]
	where the unlabeled maps are pull backs induced by the inclusions. 
	Consider 
	\[
	x = \alpha_1\smile\dots\smile \alpha_{n-2}\in H^{n-2}(M^n;\ZZ).
	\]
	Since $ e|_{M^n} = \alpha_{n-1}$, the image of $x$ in $H^{n-1}(M^n;\ZZ)$ equals 
	\[
	\alpha_1\smile\dots\smile \alpha_{n-2}\smile \alpha_{n-1},
	\]
	which is non-zero by the assumptions. Therefore, the image of $x$ in $H^{n-2}(\nu\Sigma^{n-1};\ZZ)$ is non-zero. Since $\nu\Sigma^{n-1}$ deformation retracts to $\Sigma^{n-1}$, the lemma is proved.
\end{proof}

Now we finish the proof of Theorem \ref{theo.cohomology}.
\begin{proof}[Proof of Theorem \ref{theo.cohomology}]
	By Lemma \ref{lem_Sigma_coh_cup}, $\Sigma^{n-1}$ satisfies all the assumptions of Theorem \ref{theo.cohomology}. By the induction hypothesis, $\Sigma^{n-1}$ cannot admit a metric with positive scalar curvature. This contradicts the earlier result that $g|_{\Sigma^{n-1}}$ is conformal to a positive scalar curvature metric. 
\end{proof}

\section{The enlargeability criterion and band width estimates}
\label{sec.enlargeability}
%!TEX root = main.tex
With the help of Theorem \ref{theo.cohomology} and Corollary \ref{cor_odd_deg_torus}, we establish an enlargeability criterion that obstructs existence of PSC metrics on non-orientable manifolds. We also obtain width estimates for $\ZZ_2$-over-torical bands in the process.

\begin{defi}\label{defi.enlargeable}
	Suppose $M^n$ is a closed, possibly non-orientable $n$--manifold. We say that $M^n$ is \emph{$\ZZ_2$--enlargeable}, if for each $\eps>0$, there exists a covering space $\hat M^n$ of $M^n$ and an $\eps$-contracting continuous map $f: \hat M^n\to S^n(1)$ to the unit sphere, such that $f$ is constant near infinity and has a nonzero $\ZZ_2$ degree.
\end{defi} 

This is an extension of the original definition of enlargeable manifolds for the orientable case by Gromov--Lawson \cite{GromovLawson80}. It is not hard to check that the notion of enlargeability for non-orientable manifolds depends only on the homotopy type of the manifold, and is preserved under taking connected sums or products. Generally, any manifold which admits a map of nonzero $\ZZ_2$ degree to a $\ZZ_2$--enlargeable manifold is itself enlargeable (cf. \cite[Theorem 5.3]{SpinGeometry}, where one can apply the same proof for the non-orientable case by replacing the coefficients with $\ZZ_2$). 

The main result of this section is the following theorem. 

\begin{theo}\label{theo.enlargeable}
	Suppose $3\le n \le 7$, and $M^n$ is a closed, possibly non-orientable, $\ZZ_2$--enlargeable manifold. Then $M^n$ does not admit any PSC Riemannian metric.
\end{theo}

\begin{rema}
	Theorem \ref{theo.enlargeable} can be regarded as a generalization of Corollary \ref{cor_odd_deg_torus}. This is because $\TT^n$ is $\ZZ_2$--enlargeable, so every manifold satisfying the conditions of Corollary \ref{cor_odd_deg_torus} is also $\ZZ_2$--enlargeable. 
\end{rema}

The proof of Theorem \ref{theo.enlargeable} is based on the extension of an argument of Cecchini--Schick \cite{CecchiniSchick2021}, where they proved the nonexistence of PSC metrics on orientable (but not necessarily spin) enlargeable manifolds using minimal hypersurfaces. We briefly sketch our argument before getting into the details of the proof.  Without loss of generality, assume that $(M,g)$ satisfies that $R_g>1$, and $(\hat M, \hat g)$ is a Riemannian covering with a map $f: \hat M\to S^n(1)$ with nonzero $\ZZ_2$ degree such that $|df|<\eps$. Here $\eps$ is to be chosen later. We will construct an open band $U$ contained in $\hat M$ such that $f: U\to \TT^{n-1}\times (0,1)$ is proper with nonzero $\ZZ_2$ degree, and the distance between the connected components of $\partial U$ is sufficiently large. These conditions enable us to use the standard $\mu$-bubble to find a two-sided and possibly non-orientable hypersurface $\Sigma\subset U$ that is conformal to a PSC manifold, but also admits a map to $\TT^{n-1}$ with nonzero $\ZZ_2$ degree, contradicting Corollary \ref{cor_odd_deg_torus}.

Now we present the details of the proof. 

\begin{proof}[Proof of Theorem \ref{theo.enlargeable}]
	
	Fix a smooth embedding $\iota: \TT^{n-1}\to S^n(1)$. Denote $A= \iota(\TT^{n-1})$. 
	Let $\varphi: \TT^{n-1}\times [0,1]\to S^n(1)$ be a smooth embedding such that 
	\[
	\varphi(x,\tfrac12) = \iota(x) \quad \text{ for all } x\in \TT^{n-1},
	\]
	and let 
	\[
	\delta = \tfrac12\dist_{S^n(1)} \big(\varphi(\TT^{n-1}\times\{0\}), \varphi(\TT^{n-1}\times\{1\})\big).
	\] 
	Let $N(A) = \imag(\varphi)$. Then $N(A)$ is a closed tubular neighborhood of $A$. 
	
	By the definition of $\delta$, there exists a smooth function 
	\[
	\rho: N(A) \to \RR_{\ge 0}
	\]
	such that $\rho= 0$ on $\varphi(\TT^{n-1}\times \{0\})$, $\rho = \delta$ on $\varphi(\TT^{n-1}\times \{1\})$, and $|d\rho|\le 1$ pointwise on $N(A)$.

	Now we fix $\eps< \tfrac{\delta}{2\pi}$ and consider a Riemannian covering $(\hat M, \hat g)$ with a map $f: \hat M\to S^n(1)$ with $|df|<\eps$. 
	Perturbing $f$ if necessarily, we may assume that $f$ is transverse to $\varphi(\TT^{n-1}\times\{0,\frac12, 1\})$. 
	If $\hat M$ is non-compact, by composing $f$ with a self-isometry of $S^n(1)$ if necessary, we may also assume that the image of $f$ near infinity is not contained in $N(A)$. 
		
	Then for $a = 0, 1$, we have that $f^{-1}(\varphi(\TT^{n-1}\times\{a\}))$ is a closed smooth submanifold of $\hat M$ that has an orientable normal bundle.	
	Note that $f^{-1}(\varphi(\TT^{n-1}\times[0,1]))$ is a smooth, codimension-0 submanifold of $\hat M$ whose boundary is  $f^{-1}(\varphi(\TT^{n-1}\times\{0,1\}))$. For each connected component $Y$ of  $f^{-1}(\varphi(\TT^{n-1}\times[0,1]))$, the $\ZZ_2$ degrees of the maps 
	\begin{equation}
		\label{eqn_deg_f_component}
	f: f^{-1}(\varphi(\TT^{n-1}\times\{a\})) \cap Y \to \varphi(\TT^{n-1}\times\{a\})
	\end{equation}
	are the same for $a=0, 1$ because the manifolds $f^{-1}(\varphi(\TT^{n-1}\times\{a\})) \cap Y$ are cobordant in $Y$. The sum of the $\ZZ_2$ degrees of \eqref{eqn_deg_f_component} for all connected components $Y$  of  $f^{-1}(\varphi(\TT^{n-1}\times[0,1]))$ is equal to the $\ZZ_2$ degree of $f: \hat M \to S^n(1)$. Hence we may find a connected component $Y$ such that the $\ZZ_2$ degrees of \eqref{eqn_deg_f_component} is nonzero. 
	
	Let $Y$ be as above, let 
	\[
	\partial_- Y = f^{-1}(\varphi(\TT^{n-1}\times\{0\})) \cap Y, \quad \partial_+ Y = f^{-1}(\varphi(\TT^{n-1}\times\{1\})) \cap Y.
	\]
	Then $Y$ is a smooth compact manifold with boundary $\partial_-Y\cup \partial_+Y$. 
	The composition function $\rho\circ f$ is a smooth function on $Y$ that is equal to $0$ on $\partial_-Y$, is equal to $\delta$ on $\partial_+Y$, and satisfies $|d(\rho\circ f)| <\tfrac{\delta}{2\pi}$ pointwise on $Y$.

	We consider the standard $\mu$-bubble functional on $Y$. Precisely, fix $\Omega_0 = f^{-1}(\TT^{n-1}\times (0,\tfrac12))\cap Y$. Among relative open subsets $\Omega$ containing $\partial_{-} Y$ and is disjoint from $\partial_+ Y$, minimize
	\[\cA(\Omega) = |\partial \Omega| - \int_Y (\chi_{\Omega} - \chi_{\Omega_0}) h,\]
	here $h$ is a smooth function defined by
	\[h = -\tan\left(\frac{\pi}{\delta} (\rho\circ f) - \frac{\pi}{2}\right).\]
	Since $h\to \mp \infty$ while approaching $\partial_\pm Y$, standard existence and regularity results (see, e.g. \cite{Zhu2021width}) imply that a minimizer $\Omega$ exists and $\Sigma: = \partial \Omega$ is a smooth hypersurface. We know that $\Sigma$ is two-sided as it is contained in the boundary of a domain, and it is cobordant to $\partial_{\pm} Y$ in $Y$. Therefore, the second variation of $\mu$-bubbles implies that
	\[\lambda_1(-\Delta_\Sigma + \frac12 R_\Sigma) \ge \frac12 \inf(R_{\hat g} - 2|\nabla h|+ h^2)>0,\]
	since we assumed that $R_{\hat g}>1$ and 
	\[2|\nabla h|\le \frac{1}{\cos^2(\frac{\pi}{\delta}(\rho\circ f)-\frac{\pi}{2})}\cdot \frac{2\pi}{\delta}|d\rho|\cdot|df|< \frac{1}{\cos^2(\frac{\pi}{\delta}(\rho\circ f)-\frac{\pi}{2})} = 1+h^2.\]
	Therefore $\Sigma$ is conformally PSC.

	Since $\Sigma$ is cobordant to $\partial_{\pm} Y$ in $Y$, the composition map
	\[\Sigma \hookrightarrow Y \xrightarrow{f} N(A) \cong  T^{n-1}\times [0,1] \xrightarrow{\text{projection}}\TT^{n-1}\]
	has the same $\ZZ_2$ degree as \eqref{eqn_deg_f_component}, which is assumed to be nonzero. Since $\Sigma$ is conformally PSC, this yields a contradiction with Corollary \ref{cor_odd_deg_torus}.
\end{proof}

With the above discussions, we can now easily extend Gromov's band width estimates to possibly non-orientable bands that admits a map with nonzero mod $2$ degree to $\TT^{n-1}\times [0,1]$.

\begin{defi}\label{defi.band.torical}
	A band is a compact, possibly non-orientable manifold $Y^n$ together with a decomposition $\partial Y = \partial_{-} Y\sqcup \partial_+ Y$, where $\partial_{\pm} Y$ are nonempty, disjoint unions of connected components of $\partial Y$. A band $Y$ is called $\ZZ_2$-over-torical, if it admits a proper continuous map
	\[f: Y\to \TT^{n-1}\times [0,1]\]
	with nonzero mod $2$ degree and satisfies $f(\partial_{-} Y) = \TT^{n-1}\times \{0\}$, $f(\partial_+ Y ) = \TT^{n-1}\times \{1\}$.
\end{defi}

As in the proof of Theorem \ref{theo.enlargeable}, we observe that Theorem \ref{theo.cohomology} and Corollary \ref{cor_odd_deg_torus} imply the following band width estimates.

\begin{prop}\label{prop.band.width}
	Assume $n\le 7$, $Y^n$ is a  $\ZZ_2$-over-torical band with a Riemannian metric $g$ satisfying $R_g\ge n(n-1)$. Then 
	\[\dist_g (\partial_{-} Y, \partial_+ Y) \le \frac{2\pi}{n}.\]
\end{prop}

\begin{proof}
	Observe that each connected component of $\partial_{\pm} Y$ is two-sided and admits a map with nonzero mod $2$ degree to $\TT^{n-1}$, and so is every hypersurface in $Y$ that is cobordant to $\partial_{\pm} Y$. Therefore, Gromov's $\mu$-bubble proof \cite{Gromov2018metric} of the over-torical band width estimates carries over verbatim to this case.
\end{proof}

\section{The $3$--dimensional case}
\label{sec.3d}
%!TEX root = main.tex
In this section, we prove that examples described by Theorem \ref{theo.example} do not exist in $3$-dimensions:

\begin{theo}\label{theo.3d}
	Let $M^3$ be a closed non-orientable three manifold. Then $M$ has a PSC metric if and only if its orientation double cover $\hat M$ does.
\end{theo}

The proof is based on Proposition \ref{prop_P2_decompose} below, which extends the prime decomposition to non-orientable $3$-manifolds. Proposition \ref{prop_P2_decompose} is probably known to experts in 3-manifold topology. See, for example,  \cite[Section 1.1, Exercies 4]{hatcher2007notes}. Here we give a brief proof since we were unable to find a direct reference. We will also sketch a second proof in a more geometric flavor that is closely related to \cite[Section 3]{BamlerLiMantoulidis}.

To streamline the exposition, we introduce the following notation.

\begin{defi}
	\label{defn_tilde_Y}
	Suppose $Y$ is a compact connected three-manifold that is possibly non-orientable. Suppose each connected component of $\partial Y$ is $S^2$ or $\RR P^2$. Define a closed orientable manifold $\widetilde{Y}$ as follows. If $Y$ is orientable, then every connected component of $\partial Y$ is $S^2$, and we let $\widetilde{Y}$ be the manifold obtained by gluing a copy of $D^3$ to each component of $\partial Y$. If $Y$ is non-orientable, let $\hat Y$ be the orientation double cover of $Y$, define $\widetilde{Y}$ to be the manifold obtained by gluing a copy of $D^3$ to each component of $\partial \hat Y$.
\end{defi}

Note that $\widetilde{Y}$ has a subset $Y^\dagger$ that is either the orientation double cover of $Y$ or diffeomorphic to $Y$. So we have a canonical isomorphism $\pi_2(Y^\dagger)\cong \pi_2(Y)$, thus the inclusion of $Y^\dagger$ in $Y$ induces a canonical map 
\begin{equation}
	\label{eqn_pi2_Y_to_tildeY}
	\pi_2(Y) \to \pi_2(\widetilde{Y}). 
\end{equation}
Since the second homotopy group is isomorphic to the second homology group of the universal cover, and since the universal cover of $\widetilde{Y}$ is obtained from the universal cover of $Y$ by filling in the boundaries with $D^3$, we note that \eqref{eqn_pi2_Y_to_tildeY} is a surjection.

\begin{prop}
	\label{prop_P2_decompose}
	Suppose $M$ is a closed, non-orientable, compact 3-manifold. There exists a finite collection of smooth, disjoint, embedded, two-sides surfaces, each diffeomorphic to $S^2$ or $\RR P^2$, such that after cutting open $M$ along the these surfaces, each component $C$ satisfies that either $\widetilde{C}$ is aspherical or $\widetilde{C}$ has a finite fundamental group. 
\end{prop}

\begin{proof}
	By a straightforward computation using the Mayer--Vietoris sequence, we know that if $Y$ is a compact connected $3$--manifold and $\Sigma\subset Y$ is an embedded two-sided surface that is diffeomorphic to $S^2$ or $\RR P^2$, then cutting $Y$ open along $\Sigma$ decreases the rank of $H^1(Y;\ZZ)$ by $1$. Assume $M$ is connected without loss of generality. Then, after cutting open $M$ along a finite collection of two-sided $S^2$'s and $\RR P^2$'s, we obtain a connected manifold $M'$ such that each embedded two-sided $S^2$ and $\RR P^2$ in $M'$ is separating. 
	
	If $M'$ admits an embedded two-sided surface $\Sigma$ such that:
	\begin{enumerate}
		\item $\Sigma$ is diffeomorphic to $S^2$ or $\RR P^2$,
		\item cutting open $M'$ yields two components $M'_1$ and $M'_2$ such that neither $\widetilde{M_1'}$ nor $\widetilde{M_2'}$ is diffeomorphic to $S^3$,
	\end{enumerate}
	we cut open $M'$ along $\Sigma$, and repeat the process. We first show that the process must end after finitely many steps. 
	
	Define $\hat {M'}$ to be the orientation double cover of $M'$ if $M'$ is non-orientable, and define $\hat{M'}$ to be two copies of $M'$ if $M'$ is orientable.  We invoke the following result that is part of the proof of the prime decomposition theorem for oriented 3-manifolds; see, for example, \cite{hatcher2007notes}, Statement ($\ast$) in the proof of Theorem 1.5: 
	\begin{lemm*}
			There exists a constant $n_0$ depending only on $\hat{M'}$, such that if $n\ge n_0$ and $\Sigma_1,\dots,\Sigma_{n}$ are disjoint embedded spheres in $\hat{M'}$, then after cutting open $\hat{M'}$ along all $\Sigma_i$'s, there is at least one component $C$ such that $\widetilde{C}\cong S^3$. 
	\end{lemm*}
	By the above lemma, if $M'$ is cut open along a collection of disjoint embedded two-sided surfaces consisting of $a$ copies of $S^2$'s and $b$ copies of $\RR P^2$'s, and if $2a+b\ge n_0$, then there exists a component $C$ in the resulting manifold such that $\widetilde{C} \cong S^3$. As a result, the process of cutting $M'$ must end after finitely many steps. 
	
	In summary, we have obtained a finite collection of disjoint embedded two-sided surfaces $\Sigma_1,\dots,\Sigma_n$ in $M'$, such that each surface is $S^2$ or $\RR P^2$, and cutting open $M'$ along the surfaces yields $n+1$ components 
	$M'_1,\dots,M'_{n+1}.$ Moreover, for each component $M_i'$, if $\Sigma\subset M_i'$ is an embedded $S^2$ or $\RR P^2$, then cutting open $M_i'$ along $\Sigma$ will yield at least one component $C$ such that $\widetilde{C}\cong S^3$. 
	
	We claim that $\pi_2(\widetilde{M_i'}) = 0$ for all $i$. Assume the contrary, let $N$ be the kernel of the map 
	$\pi_2(M_i') \to \pi_2(\widetilde{M_i'})$ given by \eqref{eqn_pi2_Y_to_tildeY}. Since \eqref{eqn_pi2_Y_to_tildeY} is a surjection, we have $N \subsetneq \pi_2(M_i')$. Note that $N$ is invariant under the $\pi_1(M_i')$--action. So, by the generalized sphere theorem (see \cite[Theorem 4.12]{hempel20223}), there exists an embedded two-sided surface $\Sigma\subset M_i'$ that is diffeomorphic to $S^2$ or $\RR P^2$ such that the image of $\pi_2(\Sigma)\to \pi_2(M_i')$ is not contained in $N$. On the other hand, by the properties of $M_i$, cutting open $M_i$ along $\Sigma$ yields at least one component $C$ such that $\widetilde{C}\cong S^3$, so the pre-image of $\Sigma$ in $\widetilde{M_i'}$ is null-homotopic, and hence $\pi_2(\Sigma)\to \pi_2(M_i')$ is contained in $N$. This gives a contradiction. As a result, we must have $\pi_2(\widetilde{M_i'}) = 0$ for all $i$.
	
	If  $\pi_1(\widetilde{M_i'})$ is infinite, then the universal cover of $\widetilde{M_i'}$ is a non-compact 3-manifold with vanishing $\pi_1$ and $\pi_2$. By the Hurewicz theorem, $\widetilde{M_i'}$ is aspherical. Therefore, each $\widetilde{M_i'}$ is either aspherical or has a finite fundamental group. 
\end{proof}

\begin{proof}[Proof of Theorem \ref{theo.3d}]
	Suppose that $\hat M$ has a PSC metric. Applying Proposition \ref{prop_P2_decompose} to $M$, we may cut $M$ along a disjoint union of embedded two-sided surfaces $\Sigma_1,\cdots, \Sigma_n$, each $\Sigma_j$ is either $S^2$ or $\RR P^2$, such that $M\setminus (\sqcup_{j=1}^n \Sigma_j) = \sqcup_{i=1}^m C_i$, and each $\widetilde C_i$ is either a spherical space form (if its fundamental group is finite), or is aspherical. Passing to the orientation double cover, $\hat M$ can be obtained by taking connected sums of copies of $\widetilde{C}_i$'s and $S^1\times S^2$'s. Since $\hat M$ has a PSC metric, we conclude from \cite{SY79} or \cite{GromovLawson80} that it does not admit any map with non-zero degree to a closed aspherical $3$-manifold. Consequently, each $\widetilde{C}_i$ is not aspherical and hence is a spherical space form. 
	
	We therefore conclude that $M$ can be obtained by performing either standard connected sums (i.e. $0$-surgeries), or $0$-orbifold surgeries at isolated orbifold points with $\ZZ_2$ symmetry, on the disjoint union of spherical space forms (that are possibly orbifolds). Since these surgeries preserve the PSC condition (see, e.g. \cite{GromovLawson80}), we conclude that $M$ has a PSC metric.
\end{proof}

\begin{rema}
	We briefly indicate a more geometric proof of Proposition \ref{prop_P2_decompose}. This proof closely aligns with the argument of \cite[Section 3]{BamlerLiMantoulidis}. In the place of the generalized sphere theorem \cite[Theorem 4.12]{hempel20223}, we invoke the geometric result of Meeks--Yau \cite{MeeksYau}: for a closed $3$-manifold $(M,g)$, there exist finitely many conformal maps $f_1,\cdots, f_n: S^2\to M$ with pairwise disjoint images, such that $\{f_j\}$ generates $\pi_2(M)$ as a $\pi_1(M)$-module, each $f_j$ minimizes area in its homotopy class, and each $f_j$ is either an embedding or a two-to-one covering to $\RR P^2$. Observe that the image of each $f_j$ is two-sided. This is obvious if $f_j$ is an embedding (as $S^2$ is simply connected); if $f_j$ has an image of $\Sigma_j = \RR P^2$, then it is also necessarily two-sided, as otherwise a tubular neighborhood of $\Sigma_j$ is diffeomorphic to $\RR P^3\setminus \{pt\}$, and hence $\Sigma_j$ would be homotopically trivial and thus contradicting its minimizing property.
	
	Consequently, the proof of \cite[Theorem 3.1]{BamlerLiMantoulidis} carries over verbatim; it also implies that each non-orientable PSC $3$-manifold can be obtained from a disjoint union of spherical space forms (that are possibly orbifolds) by performing $0$-surgeries (possibly at isolated orbifold singularities). 
\end{rema}

\bibliographystyle{amsplain}
\bibliography{bib.bib}

\providecommand{\bysame}{\leavevmode\hbox to3em{\hrulefill}\thinspace}
\providecommand{\MR}{\relax\ifhmode\unskip\space\fi MR }
% \MRhref is called by the amsart/book/proc definition of \MR.
\providecommand{\MRhref}[2]{%
  \href{http://www.ams.org/mathscinet-getitem?mr=#1}{#2}
}
\providecommand{\href}[2]{#2}
\begin{thebibliography}{10}

\bibitem{AIP2015}
Anar Akhmedov, Masashi Ishida, and B.~Doug Park, \emph{Dissolving 4-manifolds,
  covering spaces and {Y}amabe invariant}, Ann. Global Anal. Geom. \textbf{47}
  (2015), no.~3, 271--283. \MR{3318832}

\bibitem{ABG2024}
Hannah Alpert, Alexey Balitskiy, and Larry Guth, \emph{Macroscopic scalar
  curvature and codimension 2 width}, J. Topol. Anal. \textbf{16} (2024),
  no.~6, 979--987. \MR{4790653}

\bibitem{alpert20251urysonwidthcovers}
Hannah Alpert, Arka Banerjee, and Panos Papasoglu, \emph{1-{U}ryson width and
  covers}, 2025.

\bibitem{BamlerLiMantoulidis}
Richard~H. Bamler, Chao Li, and Christos Mantoulidis, \emph{Decomposing
  4-manifolds with positive scalar curvature}, Adv. Math. \textbf{430} (2023),
  Paper No. 109231, 17. \MR{4621960}

\bibitem{Bergery1983}
Berard Bergery, \emph{Scalar curvature and isometry group}, Spectra of
  Riemannian manifolds, Kaigai Publications, Tokyo, 1983, pp.~9--28.

\bibitem{CecchiniSchick2021}
Simone Cecchini and Thomas Schick, \emph{Enlargeable metrics on nonspin
  manifolds}, Proc. Amer. Math. Soc. \textbf{149} (2021), no.~5, 2199--2211.
  \MR{4232210}

\bibitem{CLbubbles}
Otis Chodosh and Chao Li, \emph{Generalized soap bubbles and the topology of
  manifolds with positive scalar curvature}, Ann. of Math. (2) \textbf{199}
  (2024), no.~2, 707--740. \MR{4713021}

\bibitem{CMS2023generic}
Otis Chodosh, Christos Mantoulidis, and Felix Schulze, \emph{Generic regularity
  for minimizing hypersurfaces in dimensions 9 and 10}, 2023.

\bibitem{CMS2024improved}
Otis Chodosh, Christos Mantoulidis, and Felix Schulze, \emph{Improved generic
  regularity of codimension-1 minimizing integral currents}, Ars Inveniendi
  Analytica (2024).

\bibitem{CMSW2025generic}
Otis Chodosh, Christos Mantoulidis, Felix Schulze, and Zhihan Wang,
  \emph{Generic regularity for minimizing hypersurfaces in dimension 11}, 2025.

\bibitem{Gromov1988width}
M.~Gromov, \emph{Width and related invariants of {R}iemannian manifolds}, no.
  163-164, 1988, On the geometry of differentiable manifolds (Rome, 1986),
  pp.~6, 93--109, 282. \MR{999972}

\bibitem{GromovLawson80}
Mikhael Gromov and H.~Blaine Lawson, Jr., \emph{Spin and scalar curvature in
  the presence of a fundamental group. {I}}, Ann. of Math. (2) \textbf{111}
  (1980), no.~2, 209--230. \MR{569070}

\bibitem{Gromov2018metric}
Misha Gromov, \emph{Metric inequalities with scalar curvature}, Geom. Funct.
  Anal. \textbf{28} (2018), no.~3, 645--726. \MR{3816521}

\bibitem{GromovHanke}
Misha Gromov and Bernhard Hanke, \emph{Torsion obstructions to positive scalar
  curvature}, SIGMA Symmetry Integrability Geom. Methods Appl. \textbf{20}
  (2024), Paper No. 069, 22. \MR{4843358}

\bibitem{Guth2010minimal}
Larry Guth, \emph{Systolic inequalities and minimal hypersurfaces}, Geom.
  Funct. Anal. \textbf{19} (2010), no.~6, 1688--1692. \MR{2594618}

\bibitem{HankeKotschickWehrheim2003}
B.~Hanke, D.~Kotschick, and J.~Wehrheim, \emph{Dissolving four-manifolds and
  positive scalar curvature}, Math. Z. \textbf{245} (2003), no.~3, 545--555.
  \MR{2021570}

\bibitem{hatcher2007notes}
Allen Hatcher, \emph{Notes on basic 3-manifold topology}, 2007.

\bibitem{hempel20223}
John Hempel, \emph{3-{M}anifolds}, vol. 349, American Mathematical Society,
  2022.

\bibitem{hudson1969piecewise}
John~FP Hudson, Julius~L Shaneson, and James Lees, \emph{Piecewise linear
  topology: {U}niversity of {C}hicago lecture notes}, Benjamin, 1969.

\bibitem{SpinGeometry}
H.~Blaine Lawson, Jr. and Marie-Louise Michelsohn, \emph{Spin geometry},
  Princeton Mathematical Series, vol.~38, Princeton University Press,
  Princeton, NJ, 1989. \MR{1031992}

\bibitem{LeBrun2003}
Claude LeBrun, \emph{Scalar curvature, covering spaces, and {S}eiberg-{W}itten
  theory}, New York J. Math. \textbf{9} (2003), 93--97. \MR{2016183}

\bibitem{LiokumovichMaximo}
Yevgeny Liokumovich and Davi Maximo, \emph{Waist inequality for 3-manifolds
  with positive scalar curvature}, Perspectives in scalar curvature. {V}ol. 2,
  World Sci. Publ., Hackensack, NJ, [2023] \copyright 2023, pp.~799--831.
  \MR{4577931}

\bibitem{MeeksYau}
William~H. Meeks, III and Shing~Tung Yau, \emph{Topology of three-dimensional
  manifolds and the embedding problems in minimal surface theory}, Ann. of
  Math. (2) \textbf{112} (1980), no.~3, 441--484. \MR{595203}

\bibitem{milnor2011differential}
John Milnor, \emph{Differential topology forty-six years later}, Notices of the
  AMS \textbf{58} (2011), no.~6.

\bibitem{Rosenberg1986conjecture}
Jonathan Rosenberg, \emph{{$C^\ast$}-algebras, positive scalar curvature, and
  the {N}ovikov conjecture. {III}}, Topology \textbf{25} (1986), no.~3,
  319--336. \MR{842428}

\bibitem{ruberman2007dirac}
Daniel Ruberman and Nikolai Saveliev, \emph{Dirac operators on manifolds with
  periodic ends}, Journal of G\"okova Geometry Topology \textbf{1} (2007),
  33--50.

\bibitem{rushing1973topological}
T~Benny Rushing, \emph{Topological embeddings}, vol.~52, Academic Press, 1973.

\bibitem{SY79structure}
R.~Schoen and S.~T. Yau, \emph{On the structure of manifolds with positive
  scalar curvature}, Manuscripta Math. \textbf{28} (1979), no.~1-3, 159--183.
  \MR{535700}

\bibitem{SY79}
R.~Schoen and Shing~Tung Yau, \emph{Existence of incompressible minimal
  surfaces and the topology of three-dimensional manifolds with nonnegative
  scalar curvature}, Ann. of Math. (2) \textbf{110} (1979), no.~1, 127--142.
  \MR{541332}

\bibitem{SY2019singularities}
Richard Schoen and Shing-Tung Yau, \emph{Positive scalar curvature and minimal
  hypersurface singularities}, Surveys in differential geometry 2019.
  {D}ifferential geometry, {C}alabi-{Y}au theory, and general relativity.
  {P}art 2, Surv. Differ. Geom., vol.~24, Int. Press, Boston, MA, [2022]
  \copyright 2022, pp.~441--480. \MR{4479726}

\bibitem{ShaYang}
Ji-Ping Sha and DaGang Yang, \emph{Positive {R}icci curvature on compact simply
  connected {$4$}-manifolds}, Differential geometry: {R}iemannian geometry
  ({L}os {A}ngeles, {CA}, 1990), Proc. Sympos. Pure Math., vol. 54, Part 3,
  Amer. Math. Soc., Providence, RI, 1993, pp.~529--538. \MR{1216644}

\bibitem{Zhu2021width}
Jintian Zhu, \emph{Width estimate and doubly warped product}, Trans. Amer.
  Math. Soc. \textbf{374} (2021), no.~2, 1497--1511. \MR{4196400}

\end{thebibliography}

\end{document}